\documentclass[twoside,11pt,reqno]{amsart}

\usepackage{amsmath,amsthm,amssymb,amstext,amsfonts,amscd}
\usepackage{graphicx}
\usepackage{multirow}
\usepackage{url}
\newtheorem{definition}{Definition}[section]
\newtheorem{theorem}{Theorem}

\newtheorem{remark}{Remark}[section]

\numberwithin{equation}{section}

\begin{document}

\title[Extended Beta function]
{Extended Beta, Hypergeometric and confluent Hypergeometric functions}

\author[{\bf N. U. Khan, M. Aman and T. Usman }]{\bf Nabiullah Khan, Mohd Aman and Talha Usman}

\address{Nabiullah Khan: Department of Applied
Mathematics, Faculty of Engineering and Technology,
      Aligarh Muslim University, Aligarh 202002, India}
 \email{nukhanmath@gmail.com}

\bigskip
\address{Talha Usman: Department of Applied
	Mathematics, Faculty of Engineering and Technology,
	Aligarh Muslim University, Aligarh 202002, India}
 \email{talhausman.maths@gmail.com}

\bigskip
\address{Mohd Aman: Department of Applied Mathematics, Faculty of Engineering and Technology, Aligarh Muslim University, Aligarh 202002, India}
 \email{mohdaman.maths@gmail.com}



\begin{abstract}  We aim to introduce a new extension of beta function and to study its important properties.~Using this definition, we introduce and investigate new extended hypergeometric and confluent hypergeometric functions. Further, some hybrid representations of this extended beta function are derived which include some well known special functions and polynomials.
\end{abstract}

\maketitle

\section{Introduction and preliminaries} \label{sec1}
Extending well known special functions have been an interesting sphere of research and several interesting extensions and generalizations for beta, hypergeometric and other functions can be found in literature [\cite{and-roy}-\cite{ozerg-altin},\cite{jabee-choi,sri-karls}] due to their tremendous applications. Following up with the investigation, we define here, a new extension of beta function and derive its integral representations, summation formula, mellin transform and some other relations. Further, we obtain beta distribution and some statistical formulas. Finally, using our definition of extended beta function $B_{p,q}^{\lambda}(\eta_{1},\eta_{2})$, we extend the defitions of hypergeometric and confluent hypergeometric functions. In the last section, we obtain some interesting connections of extended beta function with other special functions and polynomials as application of our results.\\
Throughout the paper, the letters $\mathbb{C}$, $\mathbb{R}$, $\mathbb{R}^+$ and $\mathbb{Z}_{0}^{-}$ denote the sets of complex numbers, real numbers, positive real numbers and non positive integers respectively, and let  $\mathbb{R}^+_{0}:= \mathbb{R} \cup \{0\}$.~The definitions given
below are crucial to derive results in the paper.

\vspace{.25cm}
\begin{definition}
As is well known, the gamma function $\Gamma(z)$ developed by Euler \cite{and-roy} with the intent to extend the factorials to values between the integers is defined by the definite integral
\begin{equation}\label{gamma}
\Gamma(z):=\int_{0}^{\infty}\,e^{-t}t^{z-1}\,dt\quad(\Re(z)>0).
\end{equation}
Among various generalizations of gamma function, we mention here the extended gamma function \cite{chaudh-zub} defined by Chaudhry and Zubair
\begin{equation}\label{egamma}
\Gamma_{p}(z)=:\int_{0}^{\infty}\,t^{z-1}\,exp\left(-t-\frac{p}{t}\right)\,dt\quad(\Re(p)>0).
\end{equation}
\end{definition}
\begin{definition}
Euler introduced the beta function (see \cite{and-roy}) for a pair of complex numbers $\eta_{1}$ and $\eta_{2}$ with positive real part through the integral
\begin{equation}\label{beta}
\aligned
B(\eta_{1},\eta_{2})=& \int_{0}^{1}\,t^{\eta_{1}-1}\,(1-t)^{\eta_{2}-1}dt\quad (\Re(\eta_{1})>0\,,\,\Re(\eta_{2})>0), \\
=&\frac{\Gamma(\eta_{1})\Gamma(\eta_{2})}{\Gamma(\eta_{1}+\eta_{2})}=\frac{(\eta_{1}-1)!\,(\eta_{2}-1)!}{(\eta_{1}+\eta_{2}-1)!} \quad (\eta_{1},\eta_{2}\notin \mathbb{Z}_{0}^{-}).
\endaligned
\end{equation}

\vspace{.25cm}
In 1997, Chaudhry et al. \cite{chaudh-raf} defined the extended beta function
\begin{equation}\label{expbeta}
B_{p}^{\lambda}(\eta_{1},\eta_{2})= \int_{0}^{1}\,t^{\eta_{1}-1}\,(1-t)^{\eta_{2}-1}\,\exp\left(-\frac{p}{t(1-t)}\right)dt, \quad (\Re(\eta_{1})>0\,,\,\Re(\eta_{2})>0)
\end{equation}

where $\Re(p)>0$ and parameters $\eta_{1}$ and $\eta_{2}$ are arbitray complex numbers. 

\vspace{.25cm} 
Shadab et al. \cite{jabee-choi} introduced
\begin{equation}\label{ebeta}
B_{p}^{\lambda}(\eta_{1},\eta_{2})= \int_{0}^{1}\,t^{\eta_{1}-1}\,(1-t)^{\eta_{2}-1}\,E_{\lambda}\left(-\frac{p}{t(1-t)}\right)dt \quad (\Re(\eta_{1})>0\,,\,\Re(\eta_{2})>0),
\end{equation}
where $E_{\lambda}(.)$ is the classical Mittag-Leffler function defined as.
\begin{equation}\label{mittag}
E_{\lambda}\left(x\right)=\sum_{n=0}^{\infty}\,\frac{x^{n}}{\Gamma(\lambda n+1)}.
\end{equation}
Note that, by putting $\lambda=1$, the above definition corresponds to the extended beta function \cite{chaudh-raf} and on putting $\lambda=1$ and $p=0$, we get the basic beta function given by \eqref{beta}.
\end{definition}

\vspace{.25cm}
\begin{definition}
The classical Gauss's hypergeometric function is defined by 
\begin{equation}\label{a3}
{}_2F_{1}{\left[\begin{array}{cccc}a,&b;\\c;\\\end{array}z\right]}=\sum_{n=0}^{\infty}\,\frac{(a)_{n}\,(b)_{n}}{(c)_{n}}\,\frac{z^{n}}{n!}={}_2F_{1}(a,b;c;\,z),
\end{equation}
where $(a)_{n}~ (a \in \mathbb{C})$ is the well known Pochhammer symbol.

\vspace{.15cm}
\noindent It is a particular case of the generalized hypergeometric series ${}_pF_{q}~(p,q \in \mathbb{N}_{0})$ defined by
\begin{equation*}
{}_pF_{q}{\left[\begin{array}{cccc}a_{1},\ldots,a_{p};\\b_{1},\ldots,b
_{q};\\\end{array}z\right]}=\sum_{n=0}^{\infty}\,\frac{(a_{1})_{n}\cdots
(a_{p})_{n}}{(b_{1})_{n}\cdots(b_{q})_{n}}\,\frac{z^{n}}{n!}.
 \end{equation*}
\vspace{.25cm}
The confluent hypergeometric function (see \cite{and-roy}) is given by the series representation
\begin{equation}\label{a4}
{}_1F_{1}(a;b;~z)
=\sum_{n=0}^{\infty}\,\frac{(a)_{n}}{(b)_{n}}\,\frac{z^{n}}{n!}.
\end{equation}

The extended hypergeometric and confluent hypergeometric functions \cite{chaudh-paris} are defined respectively by
\begin{equation}\label{ehyper}
F_{p}(\eta_{1},\eta_{2},\eta_{3};z)=\sum_{n=0}^{\infty}\,\frac{B_{p}(\eta_{2}+n\eta_{3}-\eta_{2})}{B(\eta_{2},\eta_{3}-\eta_{2})}\,\frac{z^{n}}{n!},
\end{equation}

\begin{equation*}
(p\geq0,\,\Re(\eta_{3})>\Re(\eta_{2})>0 ~and ~|z|<1)
\end{equation*}
and
\begin{equation}\label{echyper}
\Phi_{p}(\eta_{2};\eta_{3};z)=\sum_{n=0}^{\infty}\, \frac{B_{p}(\eta_{2}+n\eta_{3}-\eta_{2})}{B(\eta_{2},\eta_{3}-\eta_{2})}\,(\eta_{1})_{n}\,\frac{z^{n}}{n!},
\end{equation}
\begin{equation*}
(p\geq0 ~and ~\,\Re(\eta_{3})>\Re(\eta_{2})>0).
\end{equation*}

\vspace{.25cm}
We use the relation for Mittag-Leffler function to derive results in the paper. We have
\begin{equation}\label{rel-mittag}
\int_{0}^{\infty}\,t^{a-1}\,E_{\lambda,\gamma}^{\delta}(-wt)dt=\frac{\Gamma(a)\Gamma(\delta-a)}{\Gamma(\delta)\,w^{a}\,\Gamma(\gamma-a\lambda)},
\end{equation}
which, for $\gamma=\delta=w=1$, becomes
\begin{equation*}
\int_{0}^{\infty}\,t^{a-1}\,E_{\lambda,\gamma}^{\delta}(-wt)dt=\frac{\Gamma(a)\Gamma(1-a)}{\Gamma(1-a\lambda)}.
\end{equation*}

\end{definition}
\vspace{.35cm}
\section{A new extended Beta function and its representations}
Here, we introduce a new generalization of the extended beta function $B_{p}^{\lambda}(\eta_{1},\eta_{2})$ in \eqref{ebeta} and obtain its various properties and representations.
\begin{definition}
We define a new extension of beta function as
\begin{equation}\label{nbeta}
B_{p,q}^{\lambda;\sigma,\tau}(\eta_{1},\eta_{2})=\int_{0}^{1}\,t^{\eta_{1}-1}\,(1-t)^{\eta_{2}-1}\,E_{\lambda}\left(-\frac{p}{t^{\sigma}}\right)\,E_{\lambda}\left(-\frac{q}{(1-t)^{\tau}}\right)\,dt,
\end{equation}
\begin{equation*}
where~\Re(\eta_{1})>0,~\Re(\eta_{2})>0,~\Re(p)\geq0,~\Re(q)\geq0,~\lambda,k>0~ \text{and} ~E_{\lambda} ~\text{is a Mittag-Leffler function}.
\end{equation*}
\end{definition}
\begin{remark}
For $\sigma=\tau =\lambda=1$, \eqref{nbeta} reduce to definition \cite{choi-parmar}. For $\lambda=1$, $p=q$ and $\sigma=\tau$, \eqref{nbeta} reduce to the extended beta function \cite{lee} and for $\sigma=\tau =\lambda=1$, $p=0=q$, it reduces to the classical beta function given by \eqref{beta}.
\end{remark}

\vspace{.35cm}
\begin{center}
Integral Representations of $B_{p,q}^{\lambda;\sigma,\tau}(\eta_{1},\eta_{2})$:\end{center}
\begin{theorem}\label{2t.1}
For $\Re(p)>0,~\Re(q)>0,~\Re(r)>0,~\Re(s)>0,~\Re(\eta_{1}+r)>0,~\Re(\eta_{2}+s)>0~and~\lambda>0$,~the following integral representation holds:
\begin{equation}\label{2.2}
\aligned
\int_{0}^{\infty}\,\int_{0}^{\infty}&~p^{r-1}\,q^{s-1}~B_{p,q}^{\lambda;\sigma,\tau}(\eta_{1},\eta_{2})\,dp\,dq\\
=&\frac{\pi^{2}}{\sin(\pi r)\sin(\pi s)\Gamma(1-r\lambda)\Gamma(1-s\lambda)}\,B(x+\sigma r,y+\tau s)
\endaligned
\end{equation}
\end{theorem}

\vspace{.25cm}
\begin{proof}
The proof of \eqref{2.2} is easily derivable by multiplying both sides of \eqref{nbeta} by $p^{r-1}\,q^{s-1}$ and integrating the resulting identity with respect to $p$ and $q$~$(p\geq0,~q<\infty)$.~We have,
\begin{equation}\label{2.3}
\aligned
\int_{0}^{\infty}\,\int_{0}^{\infty}~p^{r-1}\,q^{s-1}~&B_{p,q}^{\lambda;\sigma,\tau}(\eta_{1},\eta_{2})dp\,dq
=\int_{0}^{\infty}\,\int_{0}^{\infty}~p^{r-1}\,q^{s-1}\\
\times &\left[\int_{0}^{1}\,t^{\eta_{1}-1}\,(1-t)^{\eta_{2}-1}\,E_{\lambda}\left(-\frac{p}{t^{\sigma}}\right)\,E_{\lambda}\left(-\frac{q}{(1-t)^{\tau}}\right)\,dt\right]dp\,dq
\endaligned
\end{equation}
The uniform convergence in \eqref{2.3} guarentees that the order of the integrals can be interchanged. So we have
\begin{equation*}
\aligned
&\int_{0}^{\infty}\,\int_{0}^{\infty}~p^{r-1}\,q^{s-1}~B_{p,q}^{\lambda;\sigma,\tau}(\eta_{1},\eta_{2})dp\,dq\\
&=\int_{0}^{1}\,t^{\eta_{1}-1}\,(1-t)^{\eta_{2}-1}\,\left\lbrace\int_{0}^{\infty}\,p^{r-1}E_{\lambda}\left(-\frac{p}{t^{\sigma}}\right)dp.\int_{0}^{\infty}~\,q^{s-1}~E_{\lambda}\left(-\frac{q}{(1-t)^{\tau}}\right)dq\right\rbrace\,dt
\endaligned
\end{equation*}
Now the above equation can easily be simplified with the help of \eqref{rel-mittag} to yield
\begin{equation}
\aligned
\int_{0}^{\infty}\,&\int_{0}^{\infty}~p^{r-1}\,q^{s-1}~B_{p,q}^{\lambda;\sigma,\tau}(\eta_{1},\eta_{2})dp\,dq\\
&=\int_{0}^{1}\,t^{\eta_{1}+r \sigma-1}\,(1-t)^{\eta_{2}+s \tau-1}\,\left\lbrace\frac{\Gamma(r)\Gamma(s)\,\Gamma(1-r)\Gamma(1-s)}{\Gamma(1-r\lambda)\Gamma(1-s\lambda)}\right\rbrace dt,
\endaligned
\end{equation}
which on using the relation $\Gamma(x)\Gamma(1-x)=\frac{\pi}{\sin\pi x}$ leads to the result \eqref{2.2}.
\end{proof}
\begin{remark}
For $r=s=\sigma=\tau=1$, we get an interesting relationship between the ordinary Beta function and \eqref{nbeta}.~
We have,
\begin{equation*}
\int_{0}^{\infty}\,\int_{0}^{\infty}~B_{p,q}^{\lambda}(\eta_{1},\eta_{2})\,dp\,dq=B(\eta_{1}+1,\eta_{2}+1)\quad(\Re(\eta_{1})>-1,~\Re(\eta_{2})>-1).
\end{equation*}
\end{remark}

\vspace{.25cm}
\begin{theorem}\label{2t.2}
The following representations are true:
\begin{equation}\label{2.5}
 B_{p,q}^{\lambda;\sigma,\tau}(\eta_{1},\eta_{2})=2\int_{0}^{\frac{\pi}{2}}\,Cos^{2\eta_{1}-1}\theta\,Sin^{2\eta_{2}-1}\,\theta\,E_{\lambda}\left(-\frac{p}{Cos^{2\sigma}\,\theta}\right)E_{\lambda}\left(-\frac{q}{Cos^{2\tau}\,\theta}\right)\,d\theta,
\end{equation}
\begin{equation}\label{2.6}
B_{p,q}^{\lambda;\sigma,\tau}(\eta_{1},\eta_{2})=\int_{0}^{\infty}\,\frac{u^{\eta_{1}-1}}{(1+u)^{\eta_{1}+\eta_{2}}}E_{\lambda}\left(-\frac{p(1+u)^{\sigma}}{u^{\sigma}}\right)\,E_{\lambda}\left(-\frac{q}{(1+u)^{\tau}}\right)\,du,
\end{equation}

\begin{equation}\label{2.7}
\aligned
B_{p,q}^{\lambda;\sigma,\tau}&(\eta_{1},\eta_{2})=2^{1-\eta_{1}-\eta_{2}}\hspace{8.7cm}\\
\times& \int_{-1}^{1}\,(1+u)^{\eta_{1}-1}\,(1-u)^{\eta_{2}-1}E_{\lambda}\left(-\frac{2^{\sigma}p}{(1+u)^{\sigma}}\right)\,E_{\lambda}\left(-\frac{2^{\tau}q}{(1-u)^{\tau}
}\right)\,du,
\endaligned
\end{equation}
and
\begin{equation}\label{2.8}
\aligned
B_{p,q}^{\lambda;\sigma,\tau}&(\eta_{1},\eta_{2})=(c-a)^{1-\eta_{1}-\eta_{2}}\hspace{7.8cm}\\
\times& \int_{a}^{c}\,(u-a)^{\eta_{1}-1}(c-u)^{\eta_{2}-1}E_{\lambda}\left(-\frac{p(c-a)^{\sigma}}{(u-a)^{\sigma}}\right)\,E_{\lambda}\left(-\frac{q(c-a)^{\tau}}{(c-u)^{\tau}}\right)\,du.
\endaligned
\end{equation}

\vspace{.25cm}
\begin{equation*}
(\Re(p)>0,~\Re(q)>0;~p\geq0,~q\geq0;~\lambda\geq0;~\Re(\eta_{1})>0,~\Re(\eta_{2})>0).
\end{equation*}
\end{theorem}

\vspace{.25cm}
\begin{proof}
Let $t=Cos^{2}\,\theta$,~$t=\frac{u}{1+u}$,~$t=\frac{1+u}{2}$,~$t=\frac{u-a}{c-a}$ respectively in equations \eqref{nbeta}, we obtain the above representations.
\end{proof}

\vspace{.25cm}
\begin{remark}
The above results retrieve the corresponding representaions in \cite{choi-parmar} and \cite{lee} by taking $\sigma=\tau=\lambda=1$ and $p=q$, $\lambda=1$, $\sigma=\tau$ respectively. Further for $p=0=q$, $\sigma=\tau=\lambda=1$, the results reduces to some well known results for the beta function $B(\eta_{1},\eta_{2})$.
\end{remark}

\vspace{.35cm}
\section{Properties of $B_{p,q}^{\lambda;\sigma,\tau}(\eta_{1},\eta_{2})$}
In this section, we obtain some interesting relations, summation formulas and product formulas for the generalized Beta function $B_{p,q}^{\lambda;\sigma,\tau}(\eta_{1},\eta_{2})$.
\begin{theorem}\label{2t.3}
The extended beta function satisfies the following functional relation:
\begin{equation}\label{frel}
B_{p,q}^{\lambda;\sigma,\tau}(\eta_{1}+1,\eta_{2})+B_{p,q}^{\lambda;\sigma,\tau}(\eta_{1},\eta_{2}+1)=B_{p,q}^{\lambda;\sigma,\tau}(\eta_{1},\eta_{2}).
\end{equation}
\end{theorem}
\begin{proof}
Solving L.H.S. of \eqref{frel}, we get
\begin{equation*}
\aligned
B_{p,q}^{\lambda;\sigma,\tau}&(\eta_{1}+1,\eta_{2})+B_{p,q}^{\lambda;\sigma,\tau}(\eta_{1},\eta_{2}+1)\\
=&\int_{0}^{1}\,\left\lbrace\, t^{\eta_{1}}\,(1-t)^{\eta_{2}-1}+t^{\eta_{1}-1}\,(1-t)^{\eta_{2}}\,\right\rbrace~E_{\lambda}\left(-\frac{p}{t^{\sigma}}\right)\,E_{\lambda}\left(-\frac{q}{(1-t)}^{\tau}\right)\,dt,
\endaligned
\end{equation*}
and a little manipulation leads us to the desired reult.
\end{proof}

\begin{theorem}
The following summation formula is valid for $B_{p,q}^{\lambda}(\eta_{1},\eta_{2})$:
\begin{equation}\label{sformula}
B_{p,q}^{\lambda;\sigma,\tau}(\eta_{1},1-\eta_{2})=\sum_{n=0}^{\infty}\,\frac{(\eta_{2})_{n}}{n!}\,B_{p,q}^{\lambda;\sigma,\tau}(\eta_{1}+n,1)\quad(\Re(p)>0,~\Re(q)>0).
\end{equation}
\end{theorem}
\begin{proof}
To prove above result, we make use of the generalized binomial theorem
\begin{equation*}
(1-t)^{-\eta_{2}}=\sum_{n=0}^{\infty}\,(\eta_{2})_{n}\,\frac{t^{n}}{n!},\quad (|t|<1).
\end{equation*}
Therefore, \eqref{nbeta} can be written as
\begin{equation*}
B_{p,q}^{\lambda;\sigma,\tau}(\eta_{1},1-\eta_{2})=\int_{0}^{1}\,\sum_{n=0}^{\infty}(\eta_{2})_{n}\,\frac{t^{\eta_{1}+n-1}}{n!}\,E_{\lambda}\left(-\frac{p}{t^{\sigma}}\right)\,E_{\lambda}\left(-\frac{q}{(1-t)^{\tau}}\right)\,dt.
\end{equation*}
Now by interchanging the order of integration and summation, we can easily obtain \eqref{sformula}.
\end{proof}

\begin{theorem}
The following infinite summation formula holds:
\begin{equation}\label{isformula}
B_{p,q}^{\lambda;\sigma,\tau}(\eta_{1},\eta_{2})=\sum_{n=0}^{\infty}\,B_{p,q}^{\lambda;\sigma,\tau}(\eta_{1}+n,\eta_{2}+1)\quad(\Re(p)>0, \Re(q)>0).
\end{equation}
\end{theorem}
\begin{proof}
Using the relation
\begin{equation*}
(1-t)^{\eta_{2}-1}=(1-t)^{\eta_{2}}\,\sum_{n=0}^{\infty}\,t^{n},
\end{equation*}
We obtain
\begin{equation*}
B_{p,q}^{\lambda;\sigma,\tau}(\eta_{1},\eta_{2})=\int_{0}^{1}(1-t)^{\eta_{2}}\sum_{n=0}^{\infty}\,t^{\eta_{1}+n-1}E_{\lambda}\left(-\frac{p}{t^{\sigma}}\right)\,E_{\lambda}\left(-\frac{q}{(1-t)^{\tau}}\right)\,dt.
\end{equation*}
Interchanging the order of integration and summation in the last expression leads us to the desired result \eqref{isformula}.
\end{proof}
\begin{theorem}
The following relation holds true:
\begin{equation}\label{rel}
B_{p,q}^{\lambda;\sigma,\tau}(\eta,-\eta-n)=\sum_{k=0}^{n}\,\left(\begin{tabular}{c}
n\\
k\\
\end{tabular}
\right)B_{p,q}^{\lambda;\sigma,\tau}(\eta+k,-\eta-k),\quad (n \in \mathbb{N}_{0}).
\end{equation}	 
\end{theorem}
\begin{proof}
We have
\begin{equation*}
B_{p,q}^{\lambda;\sigma,\tau}(\eta_{1}+1,\eta_{2})+B_{p,q}^{\lambda;\sigma,\tau}(\eta_{1},\eta_{2}+1)=B_{p,q}^{\lambda;\sigma,\tau}(\eta_{1},\eta_{2}).
\end{equation*}
On substituting $\eta_{1}=\eta$ and $\eta_{2}=-\eta-n$ above, we arrive at
\begin{equation*}
\hspace{1.9cm}B_{p,q}^{\lambda;\sigma,\tau}(\eta,-\eta-n)=B_{p,q}^{\lambda;\sigma,\tau}(\eta,-\eta-n+1)+B_{p,q}^{\lambda;\sigma,\tau}(\eta+1,-\eta-n).
\end{equation*}
Writing this formula recursively with $n=1,2,3,\ldots$, we obtain
\begin{equation*}
\hspace{.5cm}B_{p,q}^{\lambda;\sigma,\tau}(\eta,-\eta-1)=B_{p,q}^{\lambda;\sigma,\tau}(\eta,-\eta)+B_{p,q}^{\lambda;\sigma,\tau}(\eta+1,-\eta-1),
\end{equation*}

\begin{equation*}
B_{p,q}^{\lambda;\sigma,\tau}(\eta,-\eta-2)=B_{p,q}^{\lambda;\sigma,\tau}(\eta,-\eta)+2B_{p,q}^{\lambda;\sigma,\tau}(\eta+1,-\eta-1)+B_{p,q}^{\lambda;\sigma,\tau}(\eta+2,-\eta-2),
\end{equation*}
and so on. By continuing the process, we arrive at \eqref{rel}.
\end{proof}

\vspace{.25cm}
\begin{theorem}
The extended beta function $B_{p,q}^{\lambda;\sigma,\tau}(x,y)$ has the following Mellin transformation formula:
\begin{equation}\label{bmellin}
M\left\lbrace B_{p,q}^{\lambda}(\eta_{1},\eta_{2});p \rightarrow r,q \rightarrow s\right\rbrace =\frac{\pi^{2}}{sin(\pi r)\,sin(\pi s)\Gamma(1-r\lambda)\Gamma(1-s\lambda)}\,B(\eta_{1}+r,\eta_{2}+s)
\end{equation}
\begin{equation*}
(\Re(r)>0,~ \Re(s)>0,~\Re(\eta_{1}+r)>0,~\Re(\eta_{2}+s)>0).
\end{equation*}
\end{theorem}
\begin{proof}
We begin by providing Euler's reflection formula, which we use later in theorem
\begin{equation}\label{refformula}
\Gamma(x)\,\Gamma(1-x)=\frac{\pi}{\sin(\pi x)}.
\end{equation}
Now applying the usual Mellin transform on \eqref{nbeta}, we get
\begin{equation*}
\aligned
M&\left\lbrace B_{p,q}^{\lambda\sigma,\tau}(x,y);p \rightarrow r,q \rightarrow s\right\rbrace\\
=&\int_{0}^{\infty}\,\int_{0}^{\infty}\,p^{r-1} q^{s-1}\left\lbrace\int_{0}^{1}\,t^{\eta_{1}-1}\,(1-t)^{\eta_{2}-1}\,E_{\lambda}\left(-\frac{p}{t^{\sigma}}\right)\,E_{\lambda}\left(-\frac{q}{(1-t)^{\tau}}\right)\,dt\right\rbrace dp\,dq.
\endaligned
\end{equation*}
Interchanging the order of integrations , we have
\begin{equation*}
\aligned
M&\left\lbrace B_{p,q}^{\lambda\sigma,\tau}(\eta_{1},\eta_{2});p \rightarrow r,q \rightarrow s\right\rbrace\\
=&\int_{0}^{1}\,t^{\eta_{1}-1}\,(1-t)^{\eta_{2}-1}
\left\lbrace \int_{0}^{\infty} p^{r-1}\,E_{\lambda}\left(-\frac{p}{t^{\sigma}}\right)dp.\int_{0}^{\infty} q^{s-1}\, E_{\lambda}\left(-\frac{q}{(1-t)^{\tau}}\right)dq \right\rbrace dt.
\endaligned
\end{equation*}
Now substituting $\frac{p}{t^{\sigma}}=u$ and $\frac{q}{(1-t)^{\tau}}=v$ above, we obtain
\begin{equation*}
\aligned
M\left\lbrace B_{p,q}^{\lambda\sigma,\tau}(\eta_{1},\eta_{2});p \rightarrow r,q \rightarrow s\right\rbrace=&\int_{0}^{1}\,t^{\eta_{1}+\sigma r-1}\,(1-t)^{\eta_{2}+\tau s-1}\\
\times&\left\lbrace \int_{0}^{\infty} u^{r-1}\,E_{\lambda}\left(-u\right)du.\int_{0}^{\infty} v^{s-1}\, E_{\lambda}\left(-v\right)dv \right\rbrace dt.
\endaligned
\end{equation*}
Since we have
\begin{equation*}
\int_{0}^{\infty}\,t^{a-1}\,E_{\lambda,\gamma}^{\delta}(-wt)\,dt=\frac{\Gamma(a)\Gamma(\delta-a)}{\Gamma(\delta)\,w^{a}\,\Gamma(\gamma-a\lambda)},
\end{equation*}
which, for $\gamma=\delta=w=1$, becomes
\begin{equation*}
\int_{0}^{\infty}\,t^{a-1}\,E_{\lambda,\gamma}^{\delta}(-wt)\,dt=\frac{\Gamma(a)\Gamma(1-a)}{\Gamma(1-a\lambda)}.\hspace{.9cm}
\end{equation*}
Using above formula, we arrive at
\begin{equation*}
M\left\lbrace B_{p,q}^{\lambda\sigma,\tau}(\eta_{1},\eta_{2});p \rightarrow r,q \rightarrow s\right\rbrace=\frac{\Gamma(r)\Gamma(1-r)}{\Gamma(1-r\lambda)}\,\frac{\Gamma(s)\Gamma(1-s)}{\Gamma(1-s\lambda)}~B_{p,q}^{\lambda;\sigma,\tau}(\eta_{1}+\sigma r,\eta_{2}+\tau s),
\end{equation*}
which on using \eqref{refformula} lead to the required result \eqref{bmellin}.
\end{proof}

\vspace{.35cm}
\noindent In statistical distribution theory, gamma and beta functions have been used extensively. We now define the beta distribution of \eqref{nbeta} and obtain its mean, variance and moment generating function.\\
For $B_{p,q}^{\lambda;\sigma,\tau}(\eta_{1},\eta_{2})$, the beta distribution is given by \begin{equation}\label{bdist}
f(t)= \left\{
\begin{array}{cl}
\frac{1}{B_{p,q}^{\lambda;\sigma,\tau}(\eta_{1},\eta_{2})}t^{\eta_{1}-1}(1-t)^{\eta_{2}-1}\,E_{\lambda}\left(-\frac{p}{t^{\sigma}}\right)\,E_{\lambda}\left(-\frac{q}{(1-t)^{\tau}}\right)\quad & (0<t<1), \\
0  & \hbox{otherwise}.
\end{array}
\right.
\end{equation}
For any real number $\nu$, we have
\begin{equation}\label{expec}
E(X^{\nu})=\frac{B_{p,q}^{\lambda;\sigma,\tau}(\eta_{1}+\nu,\eta_{2})}{B_{p,q}^{\lambda;\sigma,\tau}(\eta_{1},\eta_{2})}
\end{equation}
\begin{equation*}
(p>0,~q>0,~-\infty<\eta_{1}<\infty,~\infty<\eta_{2}<\eta_{2}).
\end{equation*}
When $\nu=1$, we get the mean as a particular case of \eqref{expec}
\begin{equation}\label{mean}
\mu=E(X)=\frac{B_{p,q}^{\lambda;\sigma,\tau}(\eta_{1}+1,\eta_{2})}{B_{p,q}^{\lambda;\sigma,\tau}(\eta_{1},\eta_{2})},
\end{equation}
and the variance of the distribution is defined by
\begin{equation}\label{variance}
\sigma^{2}=E(x)-\left\lbrace E(X)\right\rbrace^{2}=\frac{B_{p,q}^{\lambda;\sigma,\tau}(\eta_{1},\eta_{2})\,B_{p,q}^{\lambda;\sigma,\tau}(\eta_{1}+2,\eta_{2})-\left\lbrace B_{p,q}^{\lambda;\sigma,\tau}(\eta_{1}+1,\eta_{2}) \right\rbrace^{2}}{\left\lbrace B_{p,q}^{\lambda;\sigma,\tau}(\eta_{1},\eta_{2})\right\rbrace^{2}}.
\end{equation}
The moment generating function of the distribution is defined by
\begin{equation}\label{moment}
M(t)=\sum_{n=0}^{\infty}\,\frac{t^{n}}{n!}\,E(X^{n})=\frac{1}{B_{p,q}^{\lambda;\sigma,\tau}(\eta_{1},\eta_{2})}\sum_{n=0}^{\infty}\,B_{p,q}^{\lambda;\sigma,\tau}(\eta_{1}+n,\eta_{2})\frac{t^{n}}{n!}.
\end{equation}
The cumulative distribution is given by
\begin{equation}\label{cdf}
F(x)=\frac{B_{z,p,q}^{\lambda;\sigma,\tau}(\eta_{1},\eta_{2})}{B_{p,q}^{\lambda;\sigma,\tau}(\eta_{1},\eta_{2})}
\end{equation}
where
\begin{equation}\label{incompleteb}
B_{z,p,q}^{\lambda;\sigma,\tau}(\eta_{1},\eta_{2})=\int_{0}^{x}\,t^{\eta_{1}-1}(1-t)^{\eta_{2}-1}\,E_{\lambda}\left(-\frac{p}{t^{\sigma}}\right)\,E_{\lambda}\left(-\frac{q}{(1-t)^{\tau}}\right)\,dt,
\end{equation}
\begin{equation*}
(p>0,~q>0,~\lambda,\sigma,\tau>0,~-\infty<\eta_{1},\eta_{2}<\infty)
\end{equation*}
is a new extension of incomplete beta function.

\vspace{.50cm}
\section{Generalization of Extended Hypergeometric and Confluent Hypergeometric functions}

Here, we introduce a generalization of extended hypergeometric and confluent hypergeometric functions in terms of $B_{p,q}^{\lambda;\sigma,\tau}(\eta_{1},\eta_{2})$.

\vspace{.25cm} 
\begin{equation}\label{nhyper}
F_{p,q}^{\lambda;\sigma,\tau}\left(\eta_{1},\eta_{2};\eta_{3};z\right)=\sum_{n=0}^{\infty}\,(\eta_{1})
_{n}~\frac{B_{p,q}^{\lambda;\sigma,\tau}(\eta_{2}+n,\eta_{3}-\eta_{2})}{B(\eta_{2};\eta_{3}-\eta_{2})}\,\frac{z^{n}}{n!},
\end{equation}
\begin{equation*}
(p\geq0,~q\geq0,~|z|<1,~\lambda,\sigma,\tau>0,~\Re(\eta_{3})>\Re(\eta_{2})>0)
\end{equation*}
and
\begin{equation}\label{nchyper}
\Phi_{p,q}^{\lambda;\sigma,\tau}\left(\eta_{2};\eta_{3};z\right)=\sum_{n=0}^{\infty}~\frac{B_{p,q}^{\lambda;\sigma,\tau}(\eta_{2}+n,\eta_{3}-\eta_{2})}{B(\eta_{2};\eta_{3}-\eta_{2})}\,\frac{z^{n}}{n!}.
\end{equation}

\begin{equation*}
(p>0,~q>0,~\lambda>0,~\Re(\eta_{3})>\Re(\eta_{2})>0)
\end{equation*}

\vspace{.15cm}
\noindent$F_{p,q}^{\lambda;\sigma,\tau}\left(\eta_{1},\eta_{2};\eta_{3};z\right)$ and $\Phi_{p,q}^{\lambda;\sigma,\tau}\left(\eta_{1},\eta_{2};\eta_{3};z\right)$ are the further generalizations of the extended Gauss hypergeometric function and extended confluent hypergeometric function and for $p=q$ and $\lambda=\sigma=\tau=1$, they reduce to \eqref{ehyper} and \eqref{echyper} respectively.

\vspace{.25cm}
\begin{center}
Integral Representations of $F_{p,q}^{\lambda;\sigma,\tau}\left(\eta_{1},\eta_{2};\eta_{3};z\right)$ \text{and} $\Phi_{p,q}^{\lambda;\sigma,\tau}\left(\eta_{1},\eta_{2};\eta_{3};z\right)$:
\end{center}

\begin{theorem}
The following integral representations for the extended hypergeometric
$F_{p,q}^{\lambda;\sigma,\tau}\left(\eta_{1},\eta_{2};\eta_{3};z\right)$ and confluent hypergeometric function $\Phi_{p,q}^{\lambda;\sigma,\tau}\left(\eta_{1},\eta_{2};\eta_{3};z\right)$ holds true:
\begin{equation}\label{5.3}
\aligned
F_{p,q}^{\lambda;\sigma,\tau}&\left(\eta_{1},\eta_{2};\eta_{3};z\right)=\frac{1}{B(\eta_{2},\eta_{3}-\eta_{2})}\\
\times& \int_{0}^{1}~t^{\eta_{2}-1}\,(1-t)^{\eta_{3}-\eta_{2}-1}\,E_{\lambda}\left(-\frac{p}{t^{\sigma}}\right)\,E_{\lambda}\left(-\frac{q}{(1-t)^{\tau}}\right)\,\sum_{n=0}^{\infty}(\eta_{1})_{n}\frac{(zt)^{n}}{n!}\,dt,
\endaligned
\end{equation}
\begin{equation*}
(p>0,~q>0;~\lambda>0;~p=0,\,q=0 ~and ~|z|<1;~\Re(\eta_{3})>\Re(\eta_{2})>0).
\end{equation*}

\begin{equation}\label{5.4}
\aligned
F_{p,q}^{\lambda;\sigma,\tau}&\left(\eta_{1},\eta_{2};\eta_{3};z\right)=\frac{1}{B(\eta_{2},\eta_{3}-\eta_{2})}\\
\times&  \int_{0}^{1}~t^{\eta_{2}-1}\,(1-t)^{\eta_{3}-\eta_{2}-1}\,(1-zt)^{-\eta_{1}}~E_{\lambda}\left(-\frac{p}{t^{\sigma}}\right)\,E_{\lambda}\left(-\frac{q}{(1-t)^{\tau}}\right)\,dt,
\endaligned
\end{equation}

\begin{equation*}
(p>0,~q>0;~\lambda>0;~p=0,\,q=0 ~and ~|\arg (1-z)|<\pi;~\Re(\eta_{3})>\Re(\eta_{2})>0).
\end{equation*}
\begin{equation}\label{5.5}
\aligned
F&_{p,q}^{\lambda;\sigma,\tau}\left(\eta_{1},\eta_{2};\eta_{3};z\right)=\frac{1}{B(\eta_{2},\eta_{3}-\eta_{2})}\\
\times&  \int_{0}^{\infty}~u^{\eta_{2}-1}\,(1+u)^{\eta_{1}-\eta_{3}}\,[u(1-z)]^{-\eta_{1}}~E_{\lambda}\left(-p\left(\frac{1+u}{u}\right)^{\sigma}\right)\,E_{\lambda}\left(-q(1+u)^{\tau}\right)\,du,
\endaligned
\end{equation}

\begin{equation*}
(p>0,~q>0;~\lambda>0;~p=0,\,q=0 ~and ~|\arg (1-z)|<\pi;~\Re(\eta_{3})>\Re(\eta_{2})>0).
\end{equation*}
\begin{equation}\label{5.6}
\aligned
F_{p,q}^{\lambda;\sigma,\tau}\left(\eta_{1},\eta_{2};\eta_{3};z\right)=&\frac{2}{B(\eta_{2},\eta_{3}-\eta_{2})}\\
\times& \int_{0}^{\frac{\pi}{2}}~\frac{\sin^{2\eta_{2}-1}v\,\cos^{2\eta_{3}-2\eta_{2}-1}\,v}{(1-z\sin^{2}v)^{\eta_{1}}}\,~E_{\lambda}\left(-p\,\csc^{2\sigma}\,v\right)\,E_{\lambda}\left(-q\sec^{2\tau}v\right)\,dv,
\endaligned
\end{equation}
\begin{equation*}
(p>0,~q>0;~\lambda>0;~p=0,\,q=0 ~and ~|\arg (1-z)|<\pi;~\Re(\eta_{3})>\Re(\eta_{2})>0).
\end{equation*}
\begin{equation}\label{5.7}
\aligned
\Phi_{p,q}^{\lambda;\sigma,\tau}\left(\eta_{2};\eta_{3};z\right)=&\frac{\exp(zt)}{B(\eta_{2},\eta_{3}-\eta_{2})}\\
\times&  \int_{0}^{1}~t^{\eta_{2}-1}\,(1-t)^{\eta_{3}-\eta_{2}-1}~E_{\lambda}\left(-\frac{p}{t^{\sigma}}\right)\,E_{\lambda}\left(-\frac{q}{(1-t)^{\tau}}\right)\,dt,
\endaligned
\end{equation}
\begin{equation*}
(p>0,~q>0;~\lambda>0;~\Re(\eta_{3})>\Re(\eta_{2})>0).
\end{equation*}

\vspace{.25cm}
\begin{equation}\label{5.8}
\aligned
\Phi_{p,q}^{\lambda;\sigma,\tau}&\left(\eta_{2};\eta_{3};z\right)=\frac{\exp(z)}{B(\eta_{2},\eta_{3}-\eta_{2})}\\
\times&  \int_{0}^{1}~t^{\eta_{2}-1}\,(1-t)^{\eta_{3}-\eta_{2}-1}~\exp(-zt)~E_{\lambda}\left(-\frac{p}{t^{\sigma}}\right)\,E_{\lambda}\left(-\frac{q}{(1-t)^{\tau}}\right)\,dt.
\endaligned
\end{equation}
\begin{equation*}
(p>0,~q>0;~\lambda>0;~\Re(\eta_{3})>\Re(\eta_{2})>0).
\end{equation*}

\end{theorem}

\vspace{.25cm}
\begin{proof}
We can easily obtain \eqref{5.3} by using the definition \eqref{nbeta} in \eqref{nhyper}. The integral \eqref{5.4} can be obtained by using the binomial expansion 
\begin{equation*}
(1-zt)^{-\eta_{1}}=\sum_{n=0}^{\infty}\,(\eta_{1})_{n}\,\frac{(zt^{n})}{n!} 
\end{equation*}
in \eqref{5.3}. By choosing $t=\frac{u}{1+u}$, $t=\sin^{2}v$ in \eqref{5.4}, we obtain \eqref{5.5} and \eqref{5.6} respectively. By using a similar approach, we can easily establish the representations \eqref{5.7} and \eqref{5.8}. 
\end{proof}

\begin{remark}
The case $\sigma=\tau =\lambda=1$ and $\lambda=1$, $p=q$, $\sigma=\tau$ in equations \eqref{5.3}-\eqref{5.8} leads to the corresponding results in \cite{choi-parmar} and \cite{lee} respectively. For $p=0=q$ and $\sigma=\tau =\lambda=1$, we get basic hypergeometric and confluent hypergeometric function \cite{and-roy}.
\end{remark}

\begin{theorem}\label{t5.8}
For $\Re(p)\geq0$,$\Re(q)\geq0$ and $\lambda\in\mathbb{C}$, the function $F_{p,q}^{\lambda;\sigma,\tau}$ possesses the following generating function:
\begin{equation}\label{gen-rel}
\aligned
\sum_{n=0}^{\infty}&\,\left( 
\begin{tabular}{c}
$\lambda$+n-1\\
n\\
\end{tabular}
\right)F_{p,q}^{\lambda;\sigma,\tau}\left(\lambda+n,\eta_{2};\eta_{3};z\right)\,t^{n}\\
=&(1-t)^{-\lambda}F_{p,q}^{\lambda;\sigma,\tau}\left(\lambda,\eta_{2};\eta_{3};\frac{z}{1-t}\right)\quad(t\leq1),
\endaligned
\end{equation}
where $F_{p,q}^{\lambda;\sigma,\tau}\left(\eta_{1},\eta_{2};\eta_{3};z\right)$ is the extended hypergeometric function defined by \eqref{nhyper}.
\end{theorem}

\vspace{.25cm}
\begin{proof}
We recall the generalized binomial coefficient (for real and complex parameters $\nu$ and $\delta$) expanded as
\begin{equation}\label{bino1}
\left( 
\begin{tabular}{c}
$\nu$\\
$\delta$\\
\end{tabular}
\right):=\frac{\Gamma(\nu+1)}{\Gamma(\delta+1)\Gamma(\nu-\delta+1)}=:\left( 
\begin{tabular}{c}
$\nu$\\
$\nu-\delta$\\
\end{tabular}
\right)\quad (\nu,\delta\in\mathbb{C}),
\end{equation}
such that for $\delta=n$~$(n\in\mathbb{N}_{0})$, we get
\begin{equation}\label{bino2}
\left( 
\begin{tabular}{c}
	$\nu$\\
	$\delta$\\
\end{tabular}
\right)=\frac{\nu(\nu-1)\cdots(\nu-n+1)}{n!}=\frac{(-1)^{n}(-\nu)_{n}}{n!}\quad(n\in\mathbb{N}_{0}).
\end{equation}
Now, let L be the left hand side of assertion \eqref{gen-rel}. Using \eqref{nhyper} into L, we get
\begin{equation}\label{4.12}
L=\sum_{n=0}^{\infty}\,\left( 
\begin{tabular}{c}
$\nu$+n-1\\
n\\
\end{tabular}
\right)\left(\sum_{k=0}^{\infty}(\nu+n)_{k}\,\frac{B_{p,q}^{\lambda;\sigma,\tau}(\eta_{2}+k,\eta_{3}-\eta_{2} )}{B(\eta_{2},\eta_{3}-\eta_{2})}\,\frac{z^{k}}{k!}\right)t^{n},
\end{equation}
which, after a little simplification, gives
\begin{equation}\label{4.13}
L=\sum_{k=0}^{\infty}\,(\nu)_{k}\,\frac{B_{p,q}^{\lambda;\sigma,\tau}(\eta_{2}+k,\eta_{3}-\eta_{2} )}{B(\eta_{2},\eta_{3}-\eta_{2})}\left[\sum_{n=0}^{\infty}\,\left( 
\begin{tabular}{c}
$\nu$+n+k-1\\
n\\
\end{tabular}
\right)t^{n}\right]\frac{z^{k}}{k!}.
\end{equation}
Finally, applying the generalized binomial expansion
\begin{equation*}
\sum_{n=0}^{\infty}\,\left( 
\begin{tabular}{c}
$\nu$+n-1\\
n\\
\end{tabular}
\right)t^{n}=(1-t)^{-\nu}\quad (|t|<1;\nu\in\mathbb{C}),
\end{equation*}
on the inner summation in \eqref{4.13}, we get the expected result \eqref{gen-rel} of Theorem \ref{t5.8}.
\end{proof}
\vspace{.35cm}
\section{Differentiation formulas for $F_{p,q}^{\lambda;\sigma,\tau}\left(\eta_{1},\eta_{2};\eta_{3};\,z\right)$ \text{and} $\Phi_{p,q}^{\lambda;\sigma,\tau}\left(\eta_{1},\eta_{2};\eta_{3};\,z\right)$}
By differentiating \eqref{nhyper} and \eqref{nchyper}, we obtain differentiation formulas with the help of the formula:
\begin{equation}\label{bdiff}
B(\eta_{2},\eta_{3}-\eta_{1})=\frac{\eta_{3}}{\eta_{2}}\,B(\eta_{2}+1,\eta_{3}-\eta_{2}).
\end{equation}
\begin{theorem}
For $n \in \mathbb{N}_{0}$, the following differentiation formulas holds true:
\begin{equation}\label{diff1}
\frac{d}{dz}\left\lbrace F_{p,q}^{\lambda;\sigma,\tau}\left(\eta_{1},\eta_{2};\eta_{3};z\right)\right\rbrace=\frac{\eta_{1}\eta_{2}}{\eta_{3}}\,F_{p,q}^{\lambda;\sigma,\tau}\left(\eta_{1}+1,\eta_{2}+1;\eta_{3}+1;\,z\right).
\end{equation}

\hskip .1cm
\begin{equation}\label{diff2}
\frac{d^{n}}{dz^{n}}\left\lbrace F_{p,q}^{\lambda;\sigma,\tau}\left(\eta_{1},\eta_{2};\eta_{3};z\right)\right\rbrace=\frac{(\eta_{1})_{n}(\eta_{2})_{n}}{(\eta_{3})_{n}}\,F_{p,q}^{\lambda;\sigma,\tau}\left(\eta_{1}+n,\eta_{2}+n;\eta_{3}+n;\,z\right).
\end{equation}

\hskip .1cm
\begin{equation}\label{diff3}
\frac{d^{n}}{dz^{n}}\left\lbrace\Phi_{p,q}^{\lambda;\sigma,\tau}\left(\eta_{1},\eta_{2};\eta_{3};z\right)\right\rbrace=\frac{(\eta_{1})_{n}(\eta_{2})_{n}}{(\eta_{3})_{n}}\,\Phi_{p,q}^{\lambda;\sigma,\tau}\left(\eta_{1}+n,\eta_{2}+n;\eta_{3}+n;\,z\right).
\end{equation}
\end{theorem}
\begin{proof}
By differentiating \eqref{nhyper} with repect to $z$, we get	
\begin{equation*}
\frac{d^{n}}{dz^{n}}\,F_{p,q}^{\lambda;\sigma,\tau}\left(\eta_{1},\eta_{2};\eta_{3};z\right)=\sum_{n=1}^{\infty}\,\frac{B_{p,q}^{\lambda;\sigma,\tau}(\eta_{2}+n,\eta_{3}-\eta_{2})}{B(\eta_{2},\eta_{3}-\eta_{2})}\,(\eta_{1})_{n}\,\frac{z^{n-1}}{(n-1)!}.
\end{equation*}
On replacing $n$ by $n+1$ and using \eqref{5.4}, we easily get \eqref{diff1}. A recursive process of this establishes \eqref{diff2}.~In a similar way, we can obtain \eqref{diff3}.
\end{proof}

\vspace{.25cm}
\begin{center}
Transformation formulas for $F_{p,q}^{\lambda;\sigma,\tau}\left(\eta_{1},\eta_{2};\eta_{3};\,z\right)$ \text{and} $\Phi_{p,q}^{\lambda;\sigma,\tau}\left(\eta_{1},\eta_{2};\eta_{3};\,z\right)$:
\end{center}
The following formulas for the extended hypergeometric and confluent hypergeometric funcion holds true:
\begin{equation}\label{trans1}
F_{p,q}^{\lambda;\sigma,\tau}\left(\eta_{1},\eta_{2};\eta_{3};\,z\right)=(1-z)^{-\alpha}\,F_{p,q}^{\lambda;\sigma,\tau}\left(\eta_{1},\eta_{3}-\eta_{2};\eta_{3};\,-\frac{z}{1-z}\right).
\end{equation}
\begin{equation*}
(p>0,q>0;\,p=0q=0 ~\text{and} ~\arg(1-z)<\pi;\,\Re(\eta_{3})>\Re(\eta_{2})>0).
\end{equation*}

\hskip .1cm
\begin{equation}\label{trans2}
F_{p,q}^{\lambda;\sigma,\tau}\left(\eta_{1},\eta_{2};\eta_{3};\,1-\frac{1}{z}\right)=z^{\alpha}\,F_{p,q}^{\lambda;\sigma,\tau}\left(\eta_{1},\eta_{3}-\eta_{2};\eta_{3};\,1-z\right).
\end{equation}
\begin{equation*}
(p>0,q>0;\,p=0q=0 ~\text{and}~ \arg(1-z)<\pi;\,\Re(\eta_{3})>\Re(\eta_{2})>0).
\end{equation*}

\hskip .1cm
\begin{equation}\label{trans3}
F_{p,q}^{\lambda;\sigma,\tau}\left(\eta_{1},\eta_{2};\eta_{3};\,\frac{z}{1+z}\right)=(1+z)^{\alpha}\,F_{p,q}^{\lambda;\sigma,\tau}\left(\eta_{1},\eta_{3}-\eta_{2};\eta_{3};\,-z\right).
\end{equation}
\begin{equation*}
(p>0,q>0;\,p=0q=0 ~\text{and} ~\arg(1-z)<\pi;\,\Re(\eta_{3})>\Re(\eta_{2})>0).
\end{equation*}

\vspace{.25cm}
\begin{equation}\label{trans4}
\Phi_{p,q}^{\lambda;\sigma,\tau}\left(\eta_{2},\eta_{3};\,z\right)=e^{z}\,\Phi_{p,q}^{\lambda;\sigma,\tau}\left(\eta_{3}-\eta_{2};\eta_{3};\,-z\right).
\end{equation}

\vspace{.35cm}
\begin{proof}
Replacing $t$ by $1-t$ in \eqref{5.4} and with the help of expression
\begin{equation*}
[1-z(1-t)]^{-\eta_{1}}=(1-z)^{-\eta_{1}}\,\left(1+\frac{z}{1-z}t\right)^{-\eta_{1}},
\end{equation*}
we have
\begin{equation}
\aligned
F_{p,q}^{\lambda;\sigma,\tau}&\left(\eta_{1},\eta_{2};\eta_{3};z\right)=\frac{(1-z)^{-\eta_{1}}}{B(\eta_{2},\eta_{3}-\eta_{2})}\\
\times&  \int_{0}^{1}~t^{\eta_{2}-1}\,(1-t)^{\eta_{3}-\eta_{2}-1}\,\left(1+\frac{z}{1-z}t\right)^{-\eta_{1}}\,E_{\lambda}\left(-\frac{p}{t^{\sigma}}\right)E_{\lambda}\left(-\frac{q}{(1-t)^{\tau}}\right)\,dt,
\endaligned
\end{equation}

\begin{equation*}
(p>0,~q>0;~\lambda>0;~p=0,\,q=0 ~and ~|\arg (1-z)|<\pi;~\Re(\eta_{3})>\Re(\eta_{2})>0).
\end{equation*}

which easily proves \eqref{trans1}. Replacing $z$ by $1-\frac{1}{z}$ and $\frac{z}{1+z}$ in \eqref{trans1} yields \eqref{trans2} and \eqref{trans3} respectively. Now the formula \eqref{trans4} can be obtained by following \eqref{5.7} and \eqref{5.8}.
\end{proof}

\vspace{.25cm}
\begin{theorem}
	The extended hypergeometric function $F_{p,q}^{\lambda;\sigma,\tau}\left(\eta_{1},\eta_{2};\eta_{3};z\right)$ has the following Mellin transformation formula:
	\begin{equation}\label{mellin-hyper}
	\aligned
	M&\left\lbrace F_{p,q}^{\lambda;\sigma,\tau}\left(\eta_{1},\eta_{2};\eta_{3};z\right);p \rightarrow r,q \rightarrow s\right\rbrace
	=\frac{\pi^{2}\,}{sin(\pi r)\,sin(\pi s)\Gamma(1-r\lambda)\Gamma(1-s\lambda)}\\
	\times &\frac{B(\eta_{2}+r,\eta_{3}+s-\eta_{2})
	}{B(\eta_{2},\eta_{3}-\eta_{2})}\,F(\eta_{1},\eta_{2}+r,\eta_{3}+r+s;z)
\endaligned
	\end{equation}
	\begin{equation*}
	(\Re(\eta_{2}+r)>0,~ \Re(\eta_{3}+s)>0).
	\end{equation*}
\end{theorem}
\begin{proof}
	We begin by providing Euler's reflection formula, which we use later in theorem
	\begin{equation}\label{euler-ref}
	\Gamma(x)\,\Gamma(1-x)=\frac{\pi}{\sin(\pi x)}.
	\end{equation}
	Applying the usual Mellin transform on \eqref{nhyper}, we get
	\begin{equation*}
	\aligned
	M&\left\lbrace F_{p,q}^{\lambda;\sigma,\tau}\left(\eta_{1},\eta_{2};\eta_{3};z\right);p \rightarrow r,q \rightarrow s\right\rbrace=\frac{1}{B(\eta_{2},\eta_{3}-\eta_{2})}\int_{0}^{\infty}\,\int_{0}^{\infty}\,p^{r-1} q^{s-1}\\
	\times&\left\lbrace\int_{0}^{1}\,t^{\eta_{2}-1}\,(1-t)^{\eta_{3}-\eta_{2}-1}\,(1-zt)^{-\eta_{1}}\, E_{\lambda}\left(-\frac{p}{t^{\sigma}}\right)E_{\lambda}\left(-\frac{q}{(1-t)^{\tau}}\right)\,dt\right\rbrace dp\,dq.
	\endaligned
	\end{equation*}
	Interchanging the order of integrations, we have
	\begin{equation*}
	\aligned
	M&\left\lbrace
	 F_{p,q}^{\lambda;\sigma,\tau}\left(\eta_{1},\eta_{2};\eta_{3};z\right);p \rightarrow r,q \rightarrow s\right\rbrace=\frac{1}{B(\eta_{2},\eta_{3}-\eta_{2})}\\
	 \times&\int_{0}^{1}\,t^{\eta_{2}-1}\,(1-t)^{\eta_{3}-\eta_{2}-1}\,(1-zt)^{-\eta_{1}}
	\left\lbrace \int_{0}^{\infty} p^{r-1}\,E_{\lambda}\left(-\frac{p}{t^{\sigma}}\right)dp.\int_{0}^{\infty} q^{s-1}\, E_{\lambda}\left(-\frac{q}{(1-t)^{\tau}}\right)dq \right\rbrace dt.
	\endaligned
	\end{equation*}
	Now substituting $\frac{p}{t^{\sigma}}=u$ and $\frac{q}{(1-t)^{\tau}}=v$ above, we obtain
	\begin{equation*}
	\aligned
	M&\left\lbrace F_{p,q}^{\lambda;\sigma,\tau}\left(\eta_{1},\eta_{2};\eta_{3};z\right);p \rightarrow r,q \rightarrow s\right\rbrace=\frac{1}{B(\eta_{2},\eta_{3}-\eta_{2})}\\
	\times&\int_{0}^{1}\,t^{\eta_{2}+r\sigma-1}\,(1-t)^{\eta_{3}+\tau s-\eta_{2}-1}\,(1-zt)^{-\eta_{1}}
	\left\lbrace \int_{0}^{\infty} u^{r-1}\,E_{\lambda}\left(-u\right)du.\int_{0}^{\infty} v^{s-1}\, E_{\lambda}\left(-v\right)dv \right\rbrace dt.
	\endaligned
	\end{equation*}
	Since we have
	\begin{equation*}
	\int_{0}^{\infty}\,t^{a-1}\,E_{\lambda,\gamma}^{\delta}(-wt)dt=\frac{\Gamma(a)\Gamma(\delta-a)}{\Gamma(\delta)\,w^{a}\,\Gamma(\gamma-a\lambda)},
	\end{equation*}
	which, for $\gamma=\delta=w=1$, becomes
	\begin{equation*}
	\int_{0}^{\infty}\,t^{a-1}\,E_{\lambda,\gamma}^{\delta}(-wt)dt=\frac{\Gamma(a)\Gamma(1-a)}{\Gamma(1-a\lambda)},
	\end{equation*}
	Using above formula, we get
	\begin{equation*}
	\aligned
	M&\left\lbrace F_{p,q}^{\lambda;\sigma,\tau}\left(\eta_{1},\eta_{2};\eta_{3};z\right);p \rightarrow r,q \rightarrow s\right\rbrace\\
	=&\frac{\Gamma(r)\Gamma(1-r)}{\Gamma(1-r\lambda)}
	\frac{\Gamma(s)\Gamma(1-s)}{\Gamma(1-s\lambda)}~\frac{B(\eta_{2}+\sigma r,\eta_{3}+\tau s-\eta_{2})}{B(\eta_{2},\eta_{3}-\eta_{2})}\,F(\eta_{1},\eta_{2}+\sigma r,\eta_{3}+\sigma r+\tau s;z).
	\endaligned
	\end{equation*}
	Taking into account the formula \eqref{euler-ref} leads us to the result \eqref{mellin-hyper}.
\end{proof}

\begin{theorem}
The following Mellin transformation formula holds:
\begin{equation}
\aligned
M&\left\lbrace \Phi_{p,q}^{\lambda;\sigma,\tau}\left(\eta_{2};\eta_{3};z\right);p \rightarrow r,q \rightarrow s\right\rbrace 
=\frac{\pi^{2}\,}{sin(\pi r)\,sin(\pi s)\Gamma(1-r\lambda)\Gamma(1-s\lambda)}
\\
\times &\frac{B(\eta_{2}+\sigma r,\eta_{3}+\tau s-\eta_{2})
}{B(\eta_{2},\eta_{3}-\eta_{2})}\,\Phi(\eta_{2}+\sigma r,\eta_{3}+\sigma r+\tau s;z).
\endaligned
\end{equation}
\begin{equation*}
(\Re(\eta_{2}+r)>0,~ \Re(\eta_{3}+s)>0).
\end{equation*}
\end{theorem}

\vspace{.35cm}
\section{Representations for $B_{p,q}^{\lambda}\,(\eta_{1},\eta_{2})$ } \label{sec7}
In this section we obtain certain connections of the generalized Beta function \eqref{nbeta} in terms of other special functions and polynomials. The results obtained here are interesting and can further be applied to other extensions of Beta and other functions.

\vspace{.25cm}
\begin{itemize}

\item (Generalized hypergeometric representation).

\vspace{.25cm}
The Mittag-Leffler function is connected to the generalized hypergeometric function (see \cite{shukla-praja}) by the relation 
\begin{equation}\label{mittag-hyper}
E_{\lambda,\beta}^{\gamma,q}(z)=\sum_{n=0}^{\infty}\,\frac{(\gamma)_{qn}}{\Gamma(\lambda n+\beta)}\,\frac{z^{n}}{n!}=\frac{1}{\Gamma(\beta)}\,{}_qF_{\lambda}\left[\varDelta(q;\gamma);\,\varDelta(\lambda,\beta);\quad \frac{q^{q}z}{\lambda^{\lambda}}\right],
\end{equation}
where, $\varDelta(\lambda,\beta)$ is a q-tuple $\frac{\gamma}{q}$, $\frac{\gamma+1}{q}$,\ldots,$\frac{\gamma+q-1}{q}$. \\
In particular, we have
\begin{equation}\label{mittag-hyper2}
E_{\lambda,1}^{1,1}(z)=E_{\lambda}(z)=\sum_{n=0}^{\infty}\,\frac{z^{n}}{\Gamma(\lambda n+1)}={}_1F_{\lambda}\left[\varDelta(1;1);\,\varDelta(\lambda,1);\quad \frac{z}{\lambda^{\lambda}}\right].
\end{equation}
Now using \eqref{mittag-hyper2} in \eqref{nbeta}, we have
\begin{equation}\label{hyperrel1}
\aligned
B_{p,q}^{\lambda;\sigma,\tau}(\eta_{1},\eta_{2})=\int_{0}^{1}\,t^{\eta_{1}-1}&(1-t)^{\eta_{2}-1}
{}_1F_{\lambda}\left[\varDelta(1;1);\,\varDelta(\lambda,1);\quad \frac{1}{\lambda^{\lambda}}\left(-\frac{p}{t^{\sigma}}\right)\right]\\
\times &{}_1F_{\lambda}\left[\varDelta(1;1);\,\varDelta(\lambda,1);\quad \frac{1}{\lambda^{\lambda}}\left(-\frac{q}{(1-t)^{\tau}}\right)\right]\,dt,
\endaligned
\end{equation}
from which we can write
\begin{equation}\label{hyperrel2}
\aligned
B_{p,q}^{\lambda;\sigma,\tau}(\eta_{1},\eta_{2})=&\int_{0}^{1}\,\frac{u^{\eta_{1}-1}}{(1+u)^{\eta_{1}+\eta_{2}}}
{}_1F_{\lambda}\left[\varDelta(1;1);\,\varDelta(\lambda,1);\quad \frac{1}{\lambda^{\lambda}}\left(-\frac{p(1+u)^{\sigma}}{u^{\sigma}}\right)\right]\\
\times&{}_1F_{\lambda}\left[\varDelta(1;1);\,\varDelta(\lambda,1);\quad \frac{1}{\lambda^{\lambda}}\left(-q(1+u)^{\tau}\right)\right]\,du.
\endaligned
\end{equation}

\vspace{.35cm}
\item($Fox\,H$-function representation)\\
We obtain the following relation between $B_{p,q}^{\lambda;\sigma,\tau}(\eta_{1},\eta_{2})$ and  Fox\,H-function:
\begin{equation}\label{Hrel}
\aligned
 B_{p,q}^{\lambda;\sigma,\tau}(\eta_{1},\eta_{2})=&\int_{0}^{1}\,t^{\eta_{1}-1}\,(1-t)^{\eta_{2}-1}
H_{0,2}^{1,0}\left[\frac{p}{t^{\sigma}}\,\huge| \quad(0;1);\,(0,1),(0,\lambda)\right]\\
\times&H_{0,2}^{1,0}\left[\frac{q}{(1-t)^{\tau}}\,\huge| \quad(0;1);\,(0,1),(0,\lambda)\right]\,dt.
\endaligned
\end{equation}

\vspace{.35cm}
\item (Bessel-Maitland function representation)\\
By using the relation (see \cite{khan-kash}) $J_{0,1}^{\lambda,1}(z)=E_{\lambda}(z)$ and in view of \eqref{nbeta},
we can write
\begin{equation}
B_{p,q}^{\lambda;\sigma,\tau}(\eta_{1},\eta_{2})=\int_{0}^{1}\,t^{\eta_{1}-1}\,(1-t)^{\eta_{2}-1}\,J_{0,1}^{\lambda,1}(u)\,J_{0,1}^{\lambda,1}(v)\,dt,
\end{equation}
where, $u=\frac{p}{t^{\sigma}}$ ~\text{and}~$v=\frac{q}{(1-t)^{\tau}}$.
\end{itemize}

\vspace{.35cm}
\section{Discussion and Conclusion}
In the present paper, it appears to hold interest that the  extensions so obtained are very general in nature and by being specific with parameters, can yield previously defined beta, hypergeometric and other hypergeometric functions. Hence they become quite important from application perspective. We have also shown the connections of the generalized beta function $B_{p,q}^{\lambda;\sigma,\tau}$ with other special functions of mathematical physics, therefore, several generating functions involving generalized (and extended) forms of beta and hypergeometric functions will likely play an essentail role in theory of applied mathematics. We also remark that the generating relation obtained in  \eqref{gen-rel} is interesting due to the fact that several functions and polynomials, for in particular, Jacobbi and Laguerre polynomials can be expressed in terms of hypergeometric and other related functions.\\
The extent to which beta and hypergeometric functions and their generalizations have contributed in mathematical physics and other fields have been a constant source of knowledge and help for researchers.
\vspace{.35cm}
\begin{center}
Acknowledgement
\end{center}
The authors would like to thank the Science and Engineering Research Board (SERB), Department of Science and Technology, Government of India (GoI) for project under the Mathematical Research Impact Centric Support (MATRICES) with reference no. MTR/2017/000821 for this work.

\end{document}